\documentclass[12pt]{amsart}

\usepackage{amsfonts, amsthm, amsmath}

\usepackage{rotating}
\usepackage{tikz}

\usepackage{amscd}

\usepackage[latin2]{inputenc}

\usepackage{t1enc}

\usepackage[mathscr]{eucal}

\usepackage{indentfirst}

\usepackage{graphicx}

\usepackage{graphics}

\usepackage{pict2e}

\usepackage{mathrsfs}

\usepackage{enumerate}
\usepackage[pagebackref]{hyperref}
\hypersetup{backref, pagebackref, colorlinks=true}
\usepackage{cite}
\usepackage{color}
\usepackage{epic}
\usepackage{hyperref} %this gives clickable references, which is nice
\numberwithin{equation}{section}
\theoremstyle{plain}

\newtheorem{theorem}{Theorem}[section]

\newtheorem{corollary}[theorem]{Corollary}

\newtheorem{proposition}[theorem]{Proposition}

\theoremstyle{definition}

\newtheorem{Def}[theorem]{Definition}

\newtheorem{example}[theorem]{Example}

\newtheorem{remark}[theorem]{Remark}

\newtheorem{?}[theorem]{Problem}

\newcommand{\N}{\mathbb{N}}

\def\G{\mathfrak{G}}

\def\D{\mathcal{D}}
\def\F{\mathcal{F}}
\def\G{\mathcal{G}}
\def\O{\mathcal{O}}
\def\P{\mathcal{P}}
\def\H{\mathcal{H}}

\def\E{\mathcal{E}}

\def\G{\mathfrak{G}}
\def\rank{\mathrm{rank}}

\def\boxit#1{\leavevmode\hbox{\vrule\vtop{\vbox{\kern.33333pt\hrule
    \kern1pt\hbox{\kern1pt\vbox{#1}\kern1pt}}\kern1pt\hrule}\vrule}}

\usepackage{collectbox}



\begin{document}

\title[Fixed largest hook length]{Partitions with fixed largest hook length}

\author[S. Fu]{Shishuo Fu}
\address[Shishuo Fu]{College of Mathematics and Statistics, Chongqing University, Huxi campus LD301, Chongqing 401331, P.R. China}
\email{fsshuo@cqu.edu.cn}

\author[D. Tang]{Dazhao Tang}

\address[Dazhao Tang]{College of Mathematics and Statistics, Chongqing University, Huxi campus LD208, Chongqing 401331, P.R. China}
\email{dazhaotang@sina.com}

\date{\today}

\begin{abstract}
Motivated by a recent paper of Straub, we study the distribution of integer partitions according to the length of their largest hook, instead of the usual statistic, namely the size of the partitions. We refine Straub's analogue of Euler's Odd-Distinct partition theorem, derive a generalization in the spirit of Alder's conjecture, as well as a curious analogue of the first Rogers-Ramanujan identity. Moreover, we obtain a partition theorem that is the counterpart of Euler's pentagonal number theory in this setting, and connect it with the Rogers-Fine identity. We concludes with some congruence properties.
\end{abstract}

\keywords{Euler's partition theorem; Euler's pentagonal number theorem; Rogers-Ramanujan identity; Rogers-Fine identity; Fibonacci number.}

\maketitle

%\tableofcontents

%%%%%%%%%%%%%%%%%%%%%%%%%%%%%%%%%%%%%
\section{Introduction}\label{sec1}
%%%%%%%%%%%%%%%%%%%%%%%%%%%%%%%%%%%%%

A \emph{partition} \cite{Andr} $\pi$ of an integer $n\in\N$ is a finite sequence of positive integers $(\pi_{1}, \pi_{2}, \ldots, \pi_{r})$ such that $\pi_{1}\geq \pi_{2}\geq\cdots\geq \pi_{r}\geq 1$ and $\pi_{1}+\pi_{2}+\cdots+\pi_{r}=n$. When $n=0$, we consider the empty partition as the only partition of $0$, and for the most part of this paper, we choose to neglect the empty partition. By convention \cite{Andr}, the integers $\pi_{1}, \pi_{2}, \cdots, \pi_{r}$ are called the \emph{parts} of $\pi$, with $\pi_{1}$ being its largest part, $\ell(\pi):=r$ the number of parts and $|\pi|:=n$ the \emph{size} of $\pi$. Such a partition $\pi$ is frequently represented by its \emph{Young diagram} (or \emph{Ferrers graph}) \cite[Chapter 1.3]{Andr}, which we take to be a left-justified array of square boxes with $r$ rows such that the $i$-th row consists of $\pi_{i}$ boxes (see Figure~\ref{yd}). This graphical representation of partitions naturally gives rise to further statistics defined on each partition. A notable one is the notion of \emph{hook length}. Each box $\nu$ is assigned a \emph{hook}, which is composed of the box $\nu$ itself as well as boxes to the right of $\nu$ (which is called the \emph{arm}, with its length denoted as $a(\nu)$) and below $\nu$ (which is called the \emph{leg}, with its length denoted as $l(\nu)$). The \emph{hook length} of $\nu$ is the number of boxes the hook consists of, i.e., $h(\nu)=a(\nu)+l(\nu)+1$. In particular, the largest hook of $\pi$, with length denoted as $\Gamma(\pi)$, traverses the leftmost column as well as the topmost row of cells, i.e.,
\begin{align}\label{Gamma}
\Gamma(\pi)=\pi_1+\ell(\pi)-1.
\end{align}
%A partition $\pi$ is said to be a \emph{$t$-core} if $\pi$ has no box of hook length equal to $t$. Similarly, a partition $\pi$ is called a \emph{($s,t$)-core} if it is at the same time an $s$-core and a $t$-core.
%%, or equivalently, it has no box of hook length with the multiples of $t$.
%The motivation for considering simultaneous cores and enumerating them is given, for example, in \cite{Ande, Arm, CEK, OS}.

%%%%%%%%%%Fig 1%%%%%%%%%%%%%%%%%%%%%%%%%
\begin{figure}
\begin{tikzpicture}[scale=0.8]
% grid
\draw (0,0) grid (1,3);
\draw (1,1) grid (2,3);
% lable
\draw (0.5,0.5) node{$1$};
\draw (0.5,1.5) node{$3$};
\draw (0.5,2.5) node{$4$};
\draw (1.5,1.5) node{$1$};
\draw (1.5,2.5) node{$2$};
\end{tikzpicture}
\caption{Young diagram for $\pi=(2,2,1)$ with hook lengths filled in}
\label{yd}
\end{figure}

In 1748, Euler \cite{Eu} proved arguably the most celebrated partition theorem. When phrased in ``partition language'', it goes as follows.
\begin{theorem}\label{od}
Given any $n\in\N$, there are as many partitions of $n$ into distinct parts as into odd parts.
\end{theorem}

In the following some three hundred years, numerous partition theorems have been discovered and stated in the form of ``there are as many partitions of $n$ satisfying condition $A$ as those satisfying condition $B$''. All of these results keep the number being partitioned, i.e., the size of the partition fixed. In a recent paper by Straub \cite{Str}, however, he chose to fix the largest hook length instead, and obtained among other things, the following nice analogue of Theorem~\ref{od}.
\begin{theorem}[Theorem 1.4 in \cite{Str}]
\label{ana-od}
The number of partitions into distinct parts with perimeter $M$ is equal to the number of partitions into odd parts with perimeter $M$. Both are enumerated by the Fibonacci number $F_M$.
\end{theorem}
Note that Straub's use of the term \emph{perimeter} was following Corteel and Lovejoy \cite[Section 4.2]{CL} (up to a shift by $1$), and it is equivalent to our largest hook length $\Gamma(\pi)$ for a given partition $\pi$.

Straub commented on his own analogue as ``missing from the literature on partitions'', and it certainly motivates us to explore further down the road and reveal more analogues with $\Gamma(\pi)$ replacing $|\pi|$, which have been unfortunately overlooked in the past. We present one of our results below, which can be viewed as an analogue of Euler's famous pentagonal number theorem \cite{Eu1}, see also \cite[Corollary~1.7]{Andr}.
\begin{Def}
For all positive integer $n$, let $\H(n)$ (resp. $h(n)$) be the set (resp. the number) of partitions $\pi$ with $\Gamma(\pi)=n$. Moreover, we consider the following subsets with respect to further restrictions, and their cardinalities will be denoted as $h_{\star}(n)$ respectively.
\begin{itemize}
\item $\H_{\D}(n):$ the set of partitions in $\H(n)$ with distinct parts;
\item $\H_{\O}(n):$ the set of partitions in $\H(n)$ with purely odd parts;
\item $\H_{\D,\O}(n):$ the set of partitions in $\H(n)$ with distinct parts and the number of parts being odd;
\item $\H_{\D,\E}(n):$ the set of partitions in $\H(n)$ with distinct parts and the number of parts being even.
\end{itemize}
\end{Def}

\begin{theorem}\label{ana-pent}
For all positive integer $n$,
\begin{align}\label{id:ana-pent}
e(n) &:=h_{\D,\E}(n)-h_{\D,\O}(n)=
\begin{cases}
0\, &n\equiv 0,3\pmod {6};\cr -1\, &n\equiv 1,2\pmod {6};\cr 1\, &n\equiv 4,5\pmod {6}.\end{cases}
\end{align}
\end{theorem}

The rest of the paper is organized as follows. In section~\ref{sec:ana-od} we study function $h(n)$ and its various relatives from scratch, deduce their generating functions combinatorially. This eventually leads to three refinements of Theorem~\ref{ana-od} as well as two generalizations. We continue to consider a natural analogue of Euler's pentagonal number theorem in section~\ref{sec:ana-pent} and establish Theorem~\ref{ana-pent}. Then Subbarao's crucial observation \cite{Sub} on Franklin's involution leads us to identity \eqref{id: Andrews} that encompasses both the pentagonal number theorem and Theorem~\ref{ana-pent}. In the last section, some congruence properties are obtained for $h_{\D}(n)$ and we conclude with some remarks.

%%%%%%%%%%%%%%%%%%%%%%%%%%%%%%%
\section{Straub's analogue of Euler's theorem}\label{sec:ana-od}
%%%%%%%%%%%%%%%%%%%%%%%%%%%%%%%

\subsection{Refinement of Straub's analogue}
%It is our believe that some results in this subsection probably have been derived in disguised form in the literature before, but we nonetheless include them here for completeness.

We use $\P$ (resp. $\D$, $\O$) to denote the set of non-empty partitions with parts being unrestricted (resp. distinct, odd).
\begin{Def}
Let $H(q):=\sum_{\pi\in\P}q^{\Gamma(\pi)}=\sum_{n=1}^{\infty}h(n)q^n$ be the generating function for $h(n)$. Similarly we define $H_{\D}(q)$ (resp. $H_{\O}(q)$) as the generating function for $h_{\D}(n)$ (resp. $h_{\O}(n)$). And we shall also consider the generating functions for the following refinements.
\begin{itemize}
\item $H(x,y,q):=\sum\limits_{\pi\in\P}x^{\pi_1}y^{\ell(\pi)}q^{\Gamma(\pi)}=\sum\limits_{n=1}^{\infty}\sum\limits_{m=1}^{n}h(m,n)x^my^{n+1-m}q^n$;
\item $H_{\D}(x,y,q):=\sum\limits_{\pi\in\D}x^{\pi_1}y^{\ell(\pi)}q^{\Gamma(\pi)}=\sum\limits_{n=1}^{\infty}\sum\limits_{m=1}^{n}h_{\D}(m,n)x^my^{n+1-m}q^n$;
\item $H_{\O}(x,y,q):=\sum\limits_{\pi\in\O}x^{\pi_1}y^{\ell(\pi)}q^{\Gamma(\pi)}=\sum\limits_{n=1}^{\infty}\sum\limits_{m=1}^{n}h_{\O}(m,n)x^my^{n+1-m}q^n$.
\end{itemize}
\end{Def}
\begin{remark}
Note that the relationship between the exponents of $x,y$ and $q$ is a direct application of \eqref{Gamma}. And clearly $H(1,1,q)=H(q)$, $H_{\D}(1,1,q)=H_{\D}(q)$ and $H_{\O}(1,1,q)=H_{\O}(q)$, with $h(n)=\sum_{m=1}^{n}h(m,n)$, $h_{\D}(n)=\sum_{m=1}^{n}h_{\D}(m,n)$ and $h_{\O}(n)=\sum_{m=1}^{n}h_{\O}(m,n)$.
\end{remark}
% Again, by convention we take the empty partition as the only partition with its largest hook of length $0$, so $h(0)=1$.
Now we proceed to derive the main results of this subsection, namely the closed form for $H(x,y,q)$, $H_{\D}(x,y,q)$ and $H_{\O}(x,y,q)$, by analysing the so-called \emph{profile} of a given partition. Following \cite{KN}, for a given partition, we use profile to denote the set of southmost and eastmost edges of the boxes in its Young diagram, as depicted in Figure~\ref{profile}. And when we label the horizontal (resp. vertical) edges in the profile with $E$ (resp. $N$), we see that the profile of any non-empty partition in $\H(n)$ is in bijection with a word consists of two letters $E$ and $N$, whose length is $n+1$ and always begins with $E$, ends with $N$. This observation essentially serves as a quick bijective proof of the fact that $h(n)=2^{n-1}$, see Corollary~\ref{h(n)}. And it sets us up for our main theorem.

%%%%%%%%%%Fig 2%%%%%%%%%%%%%%%%%%%%%%%%%
\begin{figure}
\begin{tikzpicture}[scale=1]
% grid
\draw (0,0) grid (1,3);
\draw (1,1) grid (2,3);
% lable
\draw (0.5,-0.3) node{$E$};
\draw (1.3,0.3) node{$N$};
\draw (1.6,0.8) node{$E$};
\draw (2.3,1.5) node{$N$};
\draw (2.3,2.5) node{$N$};
\draw (4,1.2) node{$\Longleftrightarrow$};
\draw (6,1.2) node{$ENENN$};
\end{tikzpicture}
\caption{The profile of $\pi=(2,2,1)$ labelled with $E$ and $N$}
\label{profile}
\end{figure}
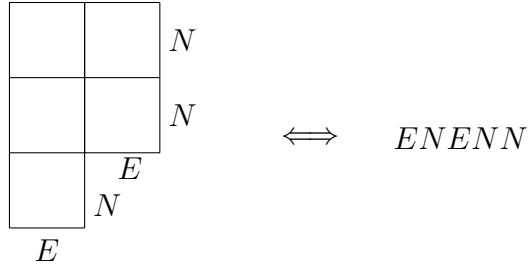

\begin{theorem}\label{gfxyq}
\begin{align}
H(x,y,q) &= xyq(1+(xq+yq)+(xq+yq)^2+\cdots)=\dfrac{xyq}{1-(xq+yq)},\label{gfH}\\
H_{\D}(x,y,q) &= xyq(1+(xq+xyq^2)+(xq+xyq^2)^2+\cdots)=\dfrac{xyq}{1-(xq+xyq^2)},\label{gfHD}\\
H_{\O}(x,y,q) &= xyq(1+(yq+x^2q^2)+(yq+x^2q^2)^2+\cdots)=\dfrac{xyq}{1-(yq+x^2q^2)}\label{gfHO}.
\end{align}
\end{theorem}
\begin{proof}
We show \eqref{gfH} first, and explain the changes need to be made for deriving \eqref{gfHD} and \eqref{gfHO}. Given any partition in $\P$, consider the word corresponding to its profile. As noted earlier, this word must begin with $E$ and end with $N$, together they contribute the factor $xyq$ (apply \eqref{Gamma} to see why it is $xyq$, not $xyq^2$). And all the remaining letters in between could be either an $E$ that contributes $xq$, or an $N$ that contributes $yq$, this explains the factor $(1+(xq+yq)+(xq+yq)^2+\cdots)$ and thus establishes \eqref{gfH}. Next for partitions in $\D$, the same restriction applies to the first and the last letters, so we also have the factor $xyq$. But now for the letters in the middle, ``having distinct parts'' translates to ``having no consecutive $N$''. So we could have either an $E$ still contributing $xq$, or a pair $NE$ contributing $xyq^2$. This justifies the factor $(1+(xq+xyq^2)+(xq+xyq^2)^2+\cdots)$ and proves \eqref{gfHD}. Finally for partitions in $\O$, all the parts are odd, which forces the corresponding word to have an even number of $E$s between two consecutive $N$, and an even number of $E$s after the first $E$. This explains the factor $(1+(x^2q^2+yq)+(x^2q^2+yq)^2+\cdots)$ and completes the proof.
\end{proof}
The next two corollaries follow immediately from Theorem~\ref{gfxyq}.
\begin{corollary}\label{h(n)}
For positive integer $n$, $h(n)=2^{n-1}$.
\end{corollary}
\begin{proof}
Put $x=y=1$ in \eqref{gfH} to get $H(q)=q/(1-2q)=\sum_{n=1}^{\infty}2^{n-1}q^n$.
\end{proof}
\begin{corollary}
Theorem~\ref{ana-od} is true.
\end{corollary}
\begin{proof}
Put $x=y=1$ in both \eqref{gfHD} and \eqref{gfHO} to get $H_{\D}(q)=H_{\O}(q)=q/(1-q-q^2)$, which is well-known as the generating function for the Fibonacci numbers.
\end{proof}
\begin{remark}
Euler's original proof of Theorem~\ref{od} used generating function. Interested readers are referred to \cite[Chapter~2,5]{AE} for a gentle introduction and \cite{Andr} for more serious discussions. Straub \cite{Str} gave two proofs of his analogue, but neither of them used generating function. So our Theorem~\ref{gfxyq} together with the above corollary can be viewed as the first generating function proof of Theorem~\ref{ana-od}.
\end{remark}

In his paper, Straub recalled two refinements \cite[Example 1.7, 1.8]{Str} of Euler's theorem due to Fine \cite{Fine}, and raised the question of seeking similar refinements for his analogue. We supply three candidates here, each of them is a direct result of evaluating $x,y$ appropriately in \eqref{gfHD} and \eqref{gfHO}.

\begin{theorem}\label{ref1}
For integers $n\geq 1$, $1\leq k \leq \lceil n/2\rceil$. The number of partitions $\pi$ into exactly $k$ parts, all distinct, and with $\Gamma(\pi)=n$ is equal to the number of partitions $\lambda$ into odd parts with $\lambda_1=2k-1$ and $\Gamma(\lambda)=n$. Both are enumerated by $\binom{n-k}{k-1}$.
\end{theorem}
\begin{proof}
Take $x=1$ in \eqref{gfHD} and $y=1$ in \eqref{gfHO}, then we see $H_{\D}(1,y,q)=\dfrac{yq}{1-(q+yq^2)}$ and $H_{\O}(x,1,q)=\dfrac{xq}{1-(q+x^2q^2)}$, both are specializations of $\dfrac{aq}{1-(q+bq^2)}$. Now we simply compare the coefficients of $ab^{k-1}q^n$ from both cases to see they are equal as claimed. Next for the exact enumeration, note that to resemble a term $ab^{k-1}q^n$, besides the numerator $aq$, we need to extract from $(k-1)+(n-1-2(k-1))=n-k$ copies of the factor $q+bq^2$, wherein $k-1$ copies we must choose $bq^2$ to guarantee $b^{k-1}$, and the remaining $n-2k+1$ copies all choose $q$ to make the total power of $q$ exactly $n$, hence the count $\binom{n-k}{k-1}$.
\end{proof}
\begin{example}
Table~\ref{tab1} lists all partitions $\pi$ into distinct parts with $\Gamma(\pi)=9, \ell(\pi)=4$, and the odd partitions $\lambda$ they get matched with such that $\Gamma(\lambda)=9, \lambda_{1}=7$. Both columns contain $\binom{9-4}{4-1}=\binom{5}{3}=10$ partitions in total.
\end{example}

%%%%%%%%%%Tab 1%%%%%%%%%%%%%%%%%%%%%%%%%
\begin{table}[htbp]
\centering \caption{}
\begin{tabular}{|c|c|}
\hline
distinct partitions & odd partitions\\
\hline
(6,5,4,3) &(7,7,7)\\
\hline
(6,5,4,2) &(7,7,5)\\
\hline
(6,5,4,1) &(7,7,3)\\
\hline
(6,5,3,2) &(7,7,1)\\
\hline
(6,5,3,1) &(7,5,5)\\
\hline
(6,5,2,1) &(7,5,3)\\
\hline
(6,4,3,2) &(7,5,1)\\
\hline
(6,4,3,1) &(7,3,3)\\
\hline
(6,4,2,1) &(7,3,1)\\
\hline
(6,3,2,1) &(7,1,1)\\
\hline
\end{tabular}
\label{tab1}
\end{table}
\begin{theorem}
For integers $n\geq 1$, $\lceil (n+1)/2\rceil\leq k \leq n$. The number of partitions $\pi$ into distinct parts, with $\pi_1=k$ and $\Gamma(\pi)=n$ is equal to the number of partitions $\lambda$ into odd parts with $\lambda_1/2+\ell(\lambda)=k+1/2$ and $\Gamma(\lambda)=n$. Both are enumerated by $\binom{k-1}{n-k}$.
\end{theorem}
\begin{proof}
Take $y=1$ in \eqref{gfHD} and $x=x^{1/2},y=x$ in \eqref{gfHO}, then we see $H_{\D}(x,1,q)=\dfrac{xq}{1-(xq+xq^2)}$ and $H_{\O}(x^{1/2},x,q)=\dfrac{x^{3/2}q}{1-(xq+xq^2)}$, which implies the statement in the theorem immediately. And we can make the exact count similarly as in the proof of Theorem~\ref{ref1}.
\end{proof}
\begin{example}
Table~\ref{tab2} lists all partitions $\pi$ into distinct parts with $\Gamma(\pi)=8, \pi_{1}=6$, and the odd partitions $\lambda$ they get matched with such that $\Gamma(\lambda)=8, \lambda_1/2+\ell(\lambda)=13/2$. Both columns contain $\binom{6-1}{8-6}=\binom{5}{2}=10$ partitions in total.
\end{example}

%%%%%%%%%%Tab 2%%%%%%%%%%%%%%%%%%%%%%%%%
\begin{table}[htbp]\caption{}
\centering
\begin{tabular}{|c|c|}
\hline
distinct partitions & odd partitions\\
\hline
(6,5,4) &(5,5,5,5)\\
\hline
(6,5,3) &(5,5,5,3)\\
\hline
(6,5,2) &(5,5,5,1)\\
\hline
(6,5,1) &(5,5,3,3)\\
\hline
(6,4,3) &(5,5,3,1)\\
\hline
(6,4,2) &(5,5,1,1)\\
\hline
(6,4,1) &(5,3,3,3)\\
\hline
(6,3,2) &(5,3,3,1)\\
\hline
(6,3,1) &(5,3,1,1)\\
\hline
(6,2,1) &(5,1,1,1)\\
\hline
\end{tabular}
\label{tab2}
\end{table}

\begin{theorem}
For integers $n\geq 1$, $0\leq k \leq n-1$. The number of partitions $\pi$ into distinct parts, with $\rank(\pi)=k$ and $\Gamma(\pi)=n$ is equal to the number of partitions $\lambda$ into odd parts with $\ell(\lambda)=k+1$ and $\Gamma(\lambda)=n$, where $\rank(\pi):=\pi_1-\ell(\pi)$. Both are enumerated by $\binom{(n+k-1)/2}{k}$. Note that $n-1$ and $k$ always have the same parity due to their definition.
\end{theorem}
\begin{proof}
Take $y=x^{-1}$ in \eqref{gfHD} and $x=1$ in \eqref{gfHO}, then we see $H_{\D}(x,x^{-1},q)=\dfrac{q}{1-(xq+q^2)}$ and $H_{\O}(1,y,q)=\dfrac{yq}{1-(yq+q^2)}$, which implies the statement in the theorem immediately. The exact count follows analogously as the previous two theorems.
\end{proof}
\begin{example}
Table~\ref{tab3} lists partitions $\pi$ into distinct parts with $\Gamma(\pi)=7, \rank(\pi)=2$, and the odd partitions $\lambda$ they get matched with such that $\Gamma(\lambda)=7, l(\lambda)=3$. Both columns contain $\binom{(7+2-1)/2}{2}=6$ partitions in total.
\end{example}

%%%%%%%%%%Tab 3%%%%%%%%%%%%%%%%%%%%%%%%%
\begin{table}[htbp]\caption{}
\centering
\begin{tabular}{|c|c|}
\hline
distinct partitions & odd partitions\\
\hline
(5,4,3) &(5,5,5)\\
\hline
(5,4,2) &(5,5,3)\\
\hline
(5,4,1) &(5,5,1)\\
\hline
(5,3,2) &(5,3,3)\\
\hline
(5,3,1) &(5,3,1)\\
\hline
(5,2,1) &(5,1,1)\\
\hline
\end{tabular}
\label{tab3}
\end{table}

\begin{remark}
All the above three refinements could be given direct bijective proofs as Straub did for Theorem~\ref{ana-od}. The idea is to correspond each $E$ step (weighted as $xq$) in distinct partition with each $N$ step (weighted as $yq$) in odd partition, while each $NE$ step (weighted as $xyq^2$) in distinct partition corresponds to each $EE$ step (weighted as $x^2q^2$) in odd partition, then keep track of how many steps of these types in them. We leave the details to interested readers.
\end{remark}

\subsection{Generalization to $d$-distinct partitions}
Following \cite{AE}, we call a partition $d$-distinct if the difference between any two of its parts is at least $d$. In 1956, Alder~\cite{Ald} investigated $q_d(n)$ and $Q_d(n)$, the number of partitions of $n$ into $d$-distinct parts and into parts $\equiv \pm 1\pmod {d+3}$, respectively, and he conjectured that
\begin{align}\label{AConj}
q_d(n)\geq Q_d(n).
\end{align}
See \cite[Chapter 4.3]{AE} for more information on this conjecture, and see \cite{AJO, Andr1, Yee} for its proof. Alder's pre-conjecture concerns us here because when $d=1,2$ we actually get ``equality'' instead of ``inequality'' in \eqref{AConj}. Indeed, the $d=1$ case is Euler's identity (Theorem~\ref{od}), and $d=2$ produces the first Rogers-Ramanujan identity. While $q_d(n)=Q_d(n)$ is too good to be true for all $d$, in our case however, we can generalize Straub's analogue (Theorem~\ref{ana-od}) to $d$-distinct partitions and get equalities for all $d$.
\begin{Def}
For positive integers $n$ and $d$, we denote $h_d(n)$ the number of $d$-distinct partitions $\pi$ with $\Gamma(\pi)=n$, and $f_d(n)$ the number of partitions $\lambda$ into parts $\equiv 1 \pmod {d+1}$ with $\Gamma(\lambda)=n$. Similarly, we denote $\H_d$ the set of all $d$-distinct partitions, and $\F_d$ the set of all partitions into parts $\equiv 1 \pmod{d+1}$.
\end{Def}
Note that with this new definition, we have $h_1(n)=h_{\D}(n), f_1(n)=h_{\O}(n)$ and $\H_1=\D,\F_1=\O$.
\begin{theorem}
For all $d\geq 1$, we have
\begin{align}
H_d(x,y,q) &:=\sum\limits_{\pi\in\H_d}x^{\pi_1}y^{\ell(\pi)}q^{\Gamma(\pi)}\label{gfHd}\\
&=xyq(1+(xq+x^dyq^{d+1})+(xq+x^dyq^{d+1})^2+\cdots)=\dfrac{xyq}{1-(xq+x^dyq^{d+1})},\nonumber \\
F_d(x,y,q) &:=\sum\limits_{\pi\in\F_d}x^{\pi_1}y^{\ell(\pi)}q^{\Gamma(\pi)}\label{gfMd}\\
&=xyq(1+(yq+x^{d+1}q^{d+1})+(yq+x^{d+1}q^{d+1})^2+\cdots)=\dfrac{xyq}{1-(yq+x^{d+1}q^{d+1})}.\nonumber
\end{align}
In particular, we have for all $n\geq 1$
\begin{align}\label{d-distinct-id}
h_d(n)=f_d(n).
\end{align}
\end{theorem}
\begin{proof}
It will suffice to show \eqref{gfHd} and \eqref{gfMd}, since \eqref{d-distinct-id} can be easily deduced from them by taking $x=y=1$ and then extracting the coefficient of $q^n$ from both. The proof of \eqref{gfHd} and \eqref{gfMd} is analogous to that of \eqref{gfHD} and \eqref{gfHO}. Given $\pi\in\H_d$, the condition of being $d$-distinct leads to the restriction on $\pi$'s profile. Indeed, the first and the last letter of the corresponding word still produces $xyq$, while the middle letters could be either an $E$ contributing $xq$, or a single $N$ followed by $d$ copies of $E$, contributing $(yq)(xq)\cdots(xq)=x^dyq^{d+1}$, this gives us \eqref{gfHd}. Similarly, for $\F_d$, to keep the parts always $\equiv 1 \pmod{d+1}$, we must have either an $N$ contributing $yq$, or consecutive $d+1$ copies of $E$ contributing $x^{d+1}q^{d+1}$, which explains the change in \eqref{gfMd} and completes the proof.
\end{proof}

When $d=2$, $2$-distinct partitions are exactly those generated by the series side of the first Rogers-Ramanujan identity, but partitions into parts $\equiv 1 \pmod 3$ (or equivalently $\equiv 1,4 \pmod 6$) do not agree with the product side of the first Rogers-Ramanujan identity, which generates partitions into parts $\equiv 1,4 \pmod 5$. Actually, a quick look at Table~\ref{tab4} reveals that with $\Gamma(\pi)\leq 5$, there are as many partitions into parts $\equiv 1,4 \pmod 6$ as there are into parts $\equiv 1,4 \pmod 5$. But when $\Gamma(\pi)\geq 6$, the former set of partitions is smaller. For instance, partition $(6)$ belongs to the latter but not to the former. This observation suggests that if one wants to force modular $5$ instead of modular $6$, some further restrictions need to be added, this is reminiscent of Schur's fix for $q_3(n)>Q_3(n)$ (see \cite[Section 4.4]{AE}). We are led to another partition theorem that involves $d$-distinct partitions.
\begin{Def}\label{defgd}
For positive integers $n$ and $d$, we denote $\G_d$ the set of all partitions $\pi=(\pi_1,\pi_2,\ldots,\pi_r)$ satisfying the following two conditions, and we denote $g_d(n)$ the number of partitions $\pi\in \G_d$ with $\Gamma(\pi)=n$.
\begin{enumerate}[i.]
\item $\pi_i\equiv 1 \text{ or } d+2 \pmod{2d+1}$, for $i=1,2,\ldots,r$;
\item $\pi_i-\pi_{i+1}\leq 2d+1$, for $i=1,2,\ldots,r$, where $\pi_{r+1}=0$, with strict sign taken whenever $\pi_i\equiv 1\pmod{2d+1}$.
\end{enumerate}
\end{Def}
\begin{theorem}\label{gen2}
For all $d\geq 1$, we have
\begin{align}
G_d(x,y,q) &:=\sum\limits_{\pi\in\G_d}x^{\pi_1}y^{\ell(\pi)}q^{\Gamma(\pi)}=\dfrac{xyq(1-yq+x^{d+1}q^{d+1})}{1-2yq+y^2q^2-x^{2d+1}yq^{2d+2}}.
\label{gfGd}
\end{align}
In particular, we have for all $n\geq 1$
\begin{align}\label{h=f=g}
h_d(n)=f_d(n)=g_d(n).
\end{align}
\end{theorem}
\begin{proof}
Take $x=y=1$ in \eqref{gfGd} gives
\begin{align*}
\sum_{n\geq 1}g_d(n)q^n=G_d(1,1,q)=\dfrac{q(1-q+q^{d+1})}{(1-q)^2-q^{2d+2}}=\dfrac{q}{1-q-q^{d+1}}=H_d(1,1,q)=F_d(1,1,q),
\end{align*}
and we get \eqref{h=f=g}. To prove \eqref{gfGd} we need to translate conditions i and ii for a given partition $\pi\in\G_d$ into restrictions on the word, say $w_{\pi}$, that corresponds to the profile of $\pi$. We claim that we can always divide $w_{\pi}$ uniquely into blocks of the following types, with all the middle blocks having ``type I, type II, type I, type II,$\ldots$'' alternately. In the following, $j=0,1,2,\ldots$ can be any natural number.
\begin{itemize}
\item Initial block: A single $E$ followed by $j$ consecutive $N$s.
\item Middle block of type I: A ($d+1$)-tuple $EE\cdots E$ followed by $j$ consecutive $N$s.
\item Middle block of type II: A ($d+1$)-tuple $NE\cdots E$ followed by $j$ consecutive $N$s.
\item Terminal block: consists of a single $N$.
\end{itemize}
Conversely, given any word that consists of these four types of blocks such that the middle blocks, if any, start with type I and switch between type I and type II, we can uniquely realize it as the profile of certain partition in $\G_d$. To prove this claim, one can trace the profile of any $\pi\in\G_d$, from the southwest-most edge to the northeast-most edge to see how $\pi$ is ``outlined'' edge by edge, and at the same time go through the word $w_{\pi}$ letter by letter from left to right. Then we see
\begin{itemize}
\item The initial block in $w_{\pi}$ corresponds to $j$ copies of $1$ as parts in $\pi$.
\item Having a middle block of type I corresponds to making the current largest part (which is necessarily $\equiv 1 \pmod{2d+1}$), say $k(2d+1)+1$, becomes $k(2d+1)+d+2$, and adding $j$ copies of it on top of the current sub-partition.
\item Having a middle block of type II corresponds to keeping the current largest part (which is necessarily $\equiv d+2 \pmod{2d+1}$), say $k(2d+1)+d+2$, and adding $j$ copies of $(k+1)(2d+1)+1$ on top of the current sub-partition.
\item The terminal block ends the process and we have a complete $\pi$, corresponding to the completed word $w_{\pi}$.
\end{itemize}
A moment of reflection reveals that in the above process of traversing the profile of $\pi$, the largest part for each sub-partition must be $\equiv 1 \pmod{2d+1}$ (resp. $\equiv d+2 \pmod{2d+1}$) after each block of type II (resp. type I), and the two middle types alternate will force condition ii, vice versa. We provide one example of $d=2$ in Figure~\ref{block}. Next we only need to figure out the generating function for each type of block.
\begin{itemize}
\item Initial$+$terminal: $xyq(1+yq+y^2q^2+\cdots)=\dfrac{xyq}{1-yq}$.
\item Type I: $x^{d+1}q^{d+1}(1+yq+y^2q^2+\cdots)=\dfrac{x^{d+1}q^{d+1}}{1-yq}$.
\item Type II: $x^dyq^{d+1}(1+yq+y^2q^2+\cdots)=\dfrac{x^{d}yq^{d+1}}{1-yq}$.
\end{itemize}
Together we have
\begin{align*}
G_d(x,y,q) &=\dfrac{xyq}{1-yq}(1+\dfrac{x^{d+1}q^{d+1}}{1-yq}+\dfrac{x^{d+1}q^{d+1}}{1-yq}\dfrac{x^{d}yq^{d+1}}{1-yq}+\dfrac{x^{d+1}q^{d+1}}{1-yq}\dfrac{x^{d}yq^{d+1}}{1-yq}\dfrac{x^{d+1}q^{d+1}}{1-yq}+\cdots)\\
&=\dfrac{\dfrac{xyq}{1-yq}(1+\dfrac{x^{d+1}q^{d+1}}{1-yq})}{1-\dfrac{x^{d+1}q^{d+1}}{1-yq}\dfrac{x^{d}yq^{d+1}}{1-yq}}=\dfrac{xyq(1-yq+x^{d+1}q^{d+1})}{1-2yq+y^2q^2-x^{2d+1}yq^{2d+2}}.
\end{align*}
This completes the proof.
\end{proof}
%%%%%%%%%%Fig 3%%%%%%%%%%%%%%%%%%%%%%%%%
\begin{figure}
\begin{tikzpicture}[scale=0.5]
% grid
\draw (0,0) grid (1,9);
\draw (1,3) grid (4,9);
\draw (4,4) grid (6,9);
\draw (6,7) grid (9,9);
\path[draw, line width=2pt] (0,0) -- (1,0) -- (1,3) -- (4,3) -- (4,4) -- (6,4) -- (6,7) -- (9,7) -- (9,9);
% lable
\draw (10,4.5) node{$\Longleftrightarrow$};
\draw (20,4.5) node{$ENNN\mid EEE\mid NEENNN\mid EEEN\mid N$};
\draw (14,6) node{{\small initial}};
\draw (17,6) node{{\small type I}};
\draw (21,6) node{{\small type II}};
\draw (25,6) node{{\small type I}};
\draw (28,6) node{{\small terminal}};
\end{tikzpicture}
\caption{Decomposition of $w_{\pi}$ into blocks for $\pi=(9,9,6,6,6,4,1,1,1)$}
\label{block}
\end{figure}
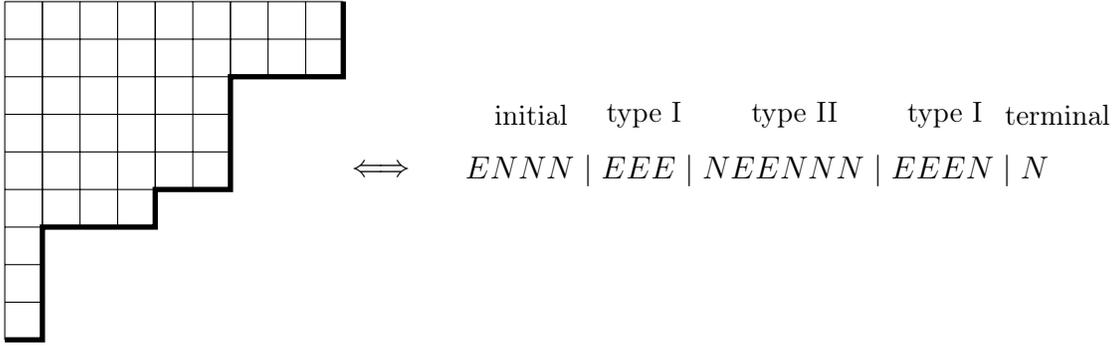
When we take $d=2$ in \eqref{h=f=g}, we get the following intriguing analogue of the first Rogers-Ramanujan identity.
\begin{corollary}
For any positive integer $n$, the number of $2$-distinct partitions with perimeter $n$ is equal to the number of partitions into parts $\equiv 1,4 \pmod 5$ with perimeter $n$ and satisfying condition ii with $d=2$ in Definition~\ref{defgd}.
\end{corollary}
\begin{example}
The partitions listed in Table \ref{tab4} all satisfy $\Gamma(\pi)\leq 7$. More precisely, we have $2$-distinct partitions in the first column, paired with partitions into parts $\equiv 1\pmod{3}$ in the second column, and with partitions into parts $\equiv 1,4 \pmod{5}$ satisfying the extra condition ii listed in the third column.
\end{example}
%%%%%%%%%%Tab 4%%%%%%%%%%%%%%%%%%%%%%%%%
\begin{table}[htbp]\caption{}
\centering
\begin{tabular}{|c|c|c|}
\hline
$2$-distinct & $\equiv 1 \pmod 3$ & condition i, ii\\
\hline
(1) & (1) & (1)\\
\hline
(2) & (1,1) & (1,1)\\
\hline
(3) &( 1,1,1) & (1,1,1)\\
\hline
(4) & (1,1,1,1) & (1,1,1,1)\\
(3,1) & (4) & (4)\\
\hline
(5) & (1,1,1,1,1) & (1,1,1,1,1)\\
(4,2) & (4,1) & (4,1)\\
(4,1) & (4,4) & (4,4)\\
\hline
(6) & (1,1,1,1,1,1) & (1,1,1,1,1,1)\\
(5,3) & (4,1,1) & (4,1,1)\\
(5,2) & (4,4,1) & (4,4,1)\\
(5,1) & (4,4,4) & (4,4,4)\\
\hline
(7) & (1,1,1,1,1,1,1) & (1,1,1,1,1,1,1)\\
(6,4) & (4,1,1,1) & (4,1,1,1)\\
(6,3) & (4,4,1,1) & (4,4,1,1)\\
(6,2) & (4,4,4,1) & (4,4,4,1)\\
(6,1) & (4,4,4,4) & (4,4,4,4)\\
(5,3,1) & (7) & (6,4)\\
\hline
\end{tabular}
\label{tab4}
\end{table}

\section{An analogue of Euler's pentagonal number theorem}\label{sec:ana-pent}
%%%%%%%%%%%%%%

\subsection{Pentagonal number theorem and a new analogue}
Straub's novel decision to fix $\Gamma(\pi)$ instead of $|\pi|$ proves successful in Theorem~\ref{ana-od}. This encourages us to seek for other analogues of the classic partition theorems with this new mindset. As mentioned in the introduction, Theorem~\ref{ana-pent} is a new partition theorem that parallels Euler's Pentagonal Number Theorem nicely. We present two proofs of it in this subsetion, and we continue to discuss yet another, purely combinatorial proof and its implications in next subsection.

\begin{proof}[1st proof of Theorem~\ref{ana-pent}]
For all partitions in $\H_{\D}(n)$, we want to calculate the number of those with even number of parts in excess of those with odd number of parts. To this end, we simply take $x=1,y=-1$ in \eqref{gfHD} and get
\begin{align*}
\sum\limits_{n=1}^{\infty} e(n)q^n &=H_{\D}(1,-1,q) = \dfrac{-q}{1-q+q^2}= \dfrac{-q}{(1-\omega q)(1-q/\omega)}\\
&=\dfrac{-1}{\sqrt{3}\:i}(\dfrac{1}{1-\omega q}-\dfrac{1}{1-q/\omega})= \dfrac{-1}{\sqrt{3}\:i}\sum\limits_{n=0}^{\infty}(\omega^n-\omega^{-n})q^n,
\end{align*}
where $\omega=e^{\pi i/3}$ is a sixth root of unity. Now a simple calculation verifies \eqref{id:ana-pent}.
\end{proof}

In preparation for our second proof of Theorem~\ref{ana-pent}, we first derive the recurrence relation for $h_{\D,\E}(n)$ and $h_{\D,\O}(n)$.
\begin{proposition}\label{rr-hDEO}
The numbers $h_{\mathcal{D, E}}(n)$ and $h_{\mathcal{D, O}}(n)$ satisfy the following system of recurrence relations, for all $n\geq 3$:
\begin{align}
h_{\mathcal{D, E}}(n) &=h_{\mathcal{D, E}}(n-1)+h_{\mathcal{D, O}}(n-2),\label{rrhDE}\\
h_{\mathcal{D, O}}(n) &=h_{\mathcal{D, O}}(n-1)+h_{\mathcal{D, E}}(n-2),\label{rrhDO}
\end{align}
with the initial values
\begin{gather}\label{iv}
h_{\D,\E}(1)=h_{\mathcal{D, E}}(2) = 0, h_{\mathcal{D, O}}(1)=h_{\mathcal{D, O}}(2)=1.\notag
\end{gather}
\end{proposition}
\begin{proof}
We show \eqref{rrhDE} combinatorially by investigating the smallest part, say $\pi_r$, in a given partition $\pi\in\H_{\D,\E}(n)$.
\begin{itemize}
\item $\pi_r=1$. Then $\pi_{r-1}\geq 2$ and we can remove the leftmost column in the Ferrers graph of $\pi$, or equivalently, subtract $1$ from each part of $\pi$. This results in a new partition $\hat{\pi}$ with one fewer parts, thus $\ell(\hat{\pi})$ is odd and $\Gamma(\hat{\pi})=n-2$. This operation is clearly seen to be invertible. Indeed, given a partition $\hat{\pi}\in\H_{\D,\O}(n-2)$, we simply add $1$ to each part of $\hat{\pi}$, as well as a new part of $1$ to arrive at a partition $\pi\in\H_{\D,\E}(n)$.
\item $\pi_r>1$. We also subtract $1$ from each part of $\pi$ to get $\hat{\pi}$. But this time $\hat{\pi}$ has the same number of parts as $\pi$, so $\ell(\hat{\pi})$ remains even and $\Gamma(\hat{\pi})=n-1$. This operation is also invertible.
\end{itemize}
Combining these two cases gives us \eqref{rrhDE}. \eqref{rrhDO} can be deduced similarly and thus omitted.
\end{proof}

It is now a routine exercise to check the following formulas for $h_{\D,\E}(n)$ and $h_{\D,\O}(n)$ using Proposition~\ref{rr-hDEO} and the Pascal relation for the binomial coefficients. So we supply here a direct combinatorial proof that do not need recurrence relations \eqref{rrhDE} and \eqref{rrhDO}.
\begin{proposition}\label{formula-hDEO}
For any positive integer $n\geq 1$,
\begin{align}
h_{\D,\E}(n) &=\sum_{k\geq0}\binom{n-2k-2}{2k+1},\label{f-hDE}\\
h_{\D,\O}(n) &=\sum_{k\geq0}\binom{n-2k-1}{2k}.\label{f-hDO}
\end{align}
\end{proposition}
\begin{proof}
We show proof of \eqref{f-hDE} only, since \eqref{f-hDO} can be proved similarly. Recall from the proof of \eqref{gfHD} that for a distinct partition $\pi$, the profile word $w_{\pi}$ can only have $E$s or couples $NE$ in the middle, with the fixed initial $E$ and the terminal $N$. And between $E$ and $NE$, only $NE$ can affect the value of $\ell(\pi)$. Therefore, if $\pi\in\H_{\D,\E}(n)$, then with the terminal $N$ accounting for the largest part, we should have an odd number of $NE$s in $w_{\pi}$, say $2k+1$ for some $k\geq 0$. Since the total length of $w_{\pi}$ is $n+1$ (see \eqref{Gamma}), we must have $n+1-2(2k+1)-1-1=n-4k-3$ copies of $E$s remained in the middle, excluding the initial $E$. Now these $n-4k-3$ copies of $E$s and $2k+1$ copies of $NE$s can be arranged in any possible order, which gives rise to $\binom{n-2k-2}{2k+1}$. We sum over all $k\geq 0$ to finish the proof.
\end{proof}

Our next proof builds on Proposition~\ref{rr-hDEO} and is free of generating function \eqref{gfHD}.

\begin{proof}[2nd proof of Theorem~\ref{ana-pent}]
For all positive integer $n$, define $e(n):=h_{\D,\E}(n)-h_{\D,\O}(n)$. Now we can subtract \eqref{rrhDO} from \eqref{rrhDE} to get $e(n)=e(n-1)-e(n-2)$, then iterate it twice we have
\begin{align*}
e(n)=-e(n-3),\quad \forall n\geq 4.
\end{align*}
Therefore $e(n)$ has a period of $6$, and it now suffice to calculate $e(1)=e(2)=-1,e(3)=0$ so as to establish \eqref{id:ana-pent} and completes the proof.
\end{proof}
\begin{remark}
Not surprisingly, we are not the first to stumble on the sequence $\{h_{\D,\E}(n)\}_{n\geq 1}$, see \href{http://oeis.org/A024490}{\tt{oeis:A024490}}. This is where we tracked Munarini and Salvi's paper \cite{MS}, in which both our Proposition~\ref{rr-hDEO} and \ref{formula-hDEO} have been derived in a quite different setting.
\end{remark}

%The values of sequences $\sigma_{n}$ are
%\begin{table}[htbp]\caption{}
%\centering
%\begin{tabular}{lcccccccccccccccccc}
%\hline
% $n$ &1 &2 &3 &4 &5 &6 &7 &8 &9 &10 &11 &12 &13 &14 &15 &16 &17 &18\\
%\hline
%$p_{\mathcal{H, D, E}}(n)$ &0 &0 &1 &2 &3 &4 &6 &10 &17 &28 &45 &72 &116 &188 &305 &494 &799 &1292\\
%\hline
%$p_{\mathcal{H, D, O}}(n)$ &1 &1 &1 &1 &2 &4 &7 &11 &17 &27 &44 &72 &117 &189 &305 &493 &798 &1292\\
%\hline
%$\sigma_{n}$ &-1 &-1 &0 &$1$ &$1$ &0 &-1 &-1 &0 &$1$ &$1$ &0 &-1 &-1 &0 &$1$ &$1$ &0\\
%\hline
%\end{tabular}
%\end{table}

\subsection{Franklin's involution}
Readers who are unfamiliar with Franklin's ingenious proof \cite{Fran} of Euler's pentagonal number theorem, are referred to \cite[Section 3.5]{AE} for an enlightening exposition. Now if one takes another look at Theorem~\ref{ana-pent} with Franklin's involution in mind, it should not take long before she/he realizes that, Franklin's involution between distinct partitions with even number of parts and those with odd number of parts not only preserve the size of the two partitions that get paired up, but also keep their perimeter unchanged! According to Andrews \cite{Andr3}, this crucial observation was first made by Subbarao \cite{Sub}, then Andrews \cite{Andr2} extended the idea considerably to give a combinatorial proof of the following Rogers--Fine identity:
\begin{align}\label{R-F}
\sum_{n=0}^{\infty}\dfrac{(\alpha)_n}{(\beta)_n}\tau^n &=\sum_{n=0}^{\infty}\dfrac{(\alpha)_n(\alpha\tau q/\beta)_n\beta^n\tau^nq^{n^2-n}(1-\alpha\tau q^{2n})}{(\beta)_n(\tau)_{n+1}}.
\end{align}
Here and in the sequel, we employ the customary notation $(a;q)_0:=1, \; (a)_n=(a;q)_n:=\prod_{k=0}^{n-1}(1-aq^k), \; \forall n\geq 1,\; (a)_{\infty}=(a;q)_{\infty}:=\prod_{k=0}^{\infty}(1-aq^k), \; |q|<1$.
One special case of \eqref{R-F} that concerns us here is the following identity we take from Andrews' paper \cite[(3.2)$\sim$(3.4)]{Andr2}:
\begin{align}\label{id: Andrews}
\sum\limits_{n=0}^{\infty}\dfrac{(-1)^ny^{2n}q^{n(n+1)/2}}{(yq)_n}&=
1+\sum\limits_{n=1}^{\infty}\sum\limits_{r=1}^{\infty}(Q_{e}(r,n)-Q_{o}(r,n))y^rq^n\\
&=1+\sum\limits_{n=1}^{\infty}(-1)^n(q^{n(3n-1)/2}y^{3n-1}+q^{n(3n+1)/2}y^{3n}),\nonumber
\end{align}
where $Q_e(r,n)$ (resp. $Q_o(r,n)$) denotes the number of distinct partitions of $n$, say $\pi$, into an even (resp. odd) number of parts such that $\pi_1+\ell(\pi)=r$.
Note that now both our Theorem~\ref{ana-pent} and Euler's pentagonal number theorem are special cases of \eqref{id: Andrews}, modulo a nuance due to \eqref{Gamma}. Indeed, given a partition $\pi$, Subbarao's original observation was on $\pi_1+\ell(\pi)$, while our perimeter $\Gamma(\pi)$ is always one less in value, and we discard entirely the empty partition that corresponds to the isolated summand $1$ in \eqref{id: Andrews}.

This new connection raises a natural question, is there a generalization of \eqref{id: Andrews} in the spirit of Sylvester's generalization \cite[(3.3)]{Andr3} of Euler's pentagonal number theorem? We give an affirmative answer with the following identity:
\begin{align}\label{Ramaphi}
\sum\limits_{n=1}^{\infty}\dfrac{x^ny^nq^{n(n+1)/2}}{(xq)_n}&=
\sum\limits_{\pi\in\D}x^{\pi_1}y^{\ell(\pi)}q^{|\pi|}=
\sum\limits_{n=1}^{\infty}\dfrac{(-yq)_{n-1}x^{2n-1}y^nq^{n(3n-1)/2}(1+xyq^{2n})}{(xq)_{n}}.
\end{align}
Note that upon taking $x=y, y=-y$ in \eqref{Ramaphi}, we get back to \eqref{id: Andrews}.
\begin{remark}
Three remarks on \eqref{Ramaphi} are in order.
\begin{itemize}
\item The first three entries from page 41 of Ramanujan's lost notebook \cite{Rama} deal with the function
\begin{align*}
\phi(a)&:=\sum_{n=0}^{\infty}\dfrac{a^nq^{n(n+1)/2}}{(bq)_n}.
\end{align*}
It is a simple matter of variable change to see the equivalence between $\phi(a)$ and the left extreme of \eqref{Ramaphi}. Consequently, \eqref{Ramaphi} is equivalent to Entry 2.6 in \cite{Rama}.
\item If we take $\beta=xq, \tau=-xyq/\alpha$, and let $\alpha\rightarrow \infty$ in \eqref{R-F}, we immediately recover the left extreme of \eqref{Ramaphi}, but the right hand side does not ``look right''. Actually this specialization gives rise to Entry 2.5 in \cite{Rama}. We recommend \cite{BY} for combinatorial proofs of all three aforementioned entries, and of many more related identities.
\item If we take $x=xq,y=yq,q=1$ in \eqref{Ramaphi} and divide both sides by $q$ to account for the ``$-1$'' in \eqref{Gamma}, we see that both extremes reduce to the same rational $\dfrac{xyq}{1-(xq+xyq^2)}$ and we recover \eqref{gfHD}.
\end{itemize}
\end{remark}

%%%%%%%%%%%%%%%%%%%%%%%%
\section{Congruence properties for $h_{\D}(n)$ and final remarks}\label{sec:cong}
%%%%%%%%%%%%%%%%%%%%%%%%
The last statement in Theorem~\ref{ana-od} links $h_{\D}(n)=h_{\O}(n)$ to the $n$-th Fibonacci number $F_n$, which we haven't explored so far. Let us first recall two properties of the Fibonacci numbers, and then use them to give congruence properties for $h_{\D}(n)$.
\begin{proposition}\label{Fm+n}
Let $m\geq 0$,$n\geq 1$ be integers, then
\begin{align}\label{id: Fm+n}
F_{m+n}=F_{m+1}F_{n}+F_{m}F_{n-1}.
\end{align}
\end{proposition}
\begin{proposition}\label{Fm|Fn}
Let $m$,$n$ be positive integers, if $m|n$, then $F_{m}|F_{n}$.
\end{proposition}
\begin{remark}
Proposition~\ref{Fm+n} can be easily proved by induction, then Proposition~\ref{Fm|Fn} follows from \eqref{id: Fm+n} combined with induction. It is worth recommending the award-winning book \cite{BQ}, where Benjamin and Quinn give an interesting tiling proof of Proposition~\ref{Fm+n}, among their masterful treatment of tons of other identities involving the Fibonacci numbers.
\end{remark}
Now the following congruence properties for $h_{\D}(n)$ should not come as surprise.
\begin{align}
h_{\mathcal{D}}(3n) &\equiv0\pmod{2}\label{3n}\\
h_{\mathcal{D}}(4n) &\equiv0\pmod{3}\\
h_{\mathcal{D}}(5n) &\equiv0\pmod{5}\\
h_{\mathcal{D}}(6n) &\equiv0 \pmod{8}\label{6n}\\
h_{\mathcal{D}}(6n+3) &\equiv2 \pmod{16}\label{6n+3}\\
h_{\mathcal{D, O}}(6n) &=h_{\mathcal{D, E}}(6n)\equiv0 \pmod{4}\label{DOE6n}\\
h_{\mathcal{D, O}}(6n+3) &=h_{\mathcal{D, E}}(6n+3)\equiv1 \pmod{8}\label{DOE6n+3}
\end{align}
Note that \eqref{3n} through \eqref{6n} require only Theorem~\ref{ana-od}, Proposition~\ref{Fm|Fn} and initial values, while \eqref{6n+3} needs more patience to wait until a complete period for $\{F_n \pmod{16}\}_{n\geq 1}$ emerges. Finally \eqref{DOE6n} and \eqref{DOE6n+3} are direct results of \eqref{6n} and \eqref{6n+3}, upon applying Theorem~\ref{ana-pent}. These results are by no means the complete list, but they raise the natural question of seeking partition statistic analogous to the rank (or crank) that would provide a combinatorial interpretation  of the above congruences. Lastly, it would certainly be interesting to also try and generalize both \eqref{gfH} and \eqref{gfHO} to the extent of \eqref{Ramaphi}.

%%%%%%%%%%%%%%%%%%%%%%%%%%
\section*{Acknowledgement}
%%%%%%%%%%%%%%%%%%%%%%%%%%
The authors are grateful to Jiang Zeng for bringing Straub's paper \cite{Str} to their attention and for some initial discussions that motivated this work. Thanks also go to Zichen Yang for suggesting a better form of Theorem~\ref{ref1}.

%And the authors also acknowledge the helpful suggestions made by the referee.

Both authors were supported by the Fundamental Research Funds for the Central Universities (No.~CQDXWL-2014-Z004) and the National Science Foundation of China (No.~115010\\61).

\end{document}